\documentclass[oneside,english]{amsart}
\usepackage[T1]{fontenc}
\usepackage{textcomp}
\usepackage[utf8]{inputenc}
\usepackage{verbatim}
\usepackage{amstext}
\usepackage{amsthm}
\usepackage{amssymb}
\usepackage{xargs}[2008/03/08]

\makeatletter
\numberwithin{equation}{section}
\theoremstyle{plain}
\newtheorem{thm}{\protect\theoremname}[section]
\theoremstyle{definition}
\newtheorem{defn}[thm]{\protect\definitionname}
\theoremstyle{remark}
\newtheorem{rem}[thm]{\protect\remarkname}
\theoremstyle{plain}
\newtheorem{lem}[thm]{\protect\lemmaname}
\theoremstyle{plain}
\newtheorem{cor}[thm]{\protect\corollaryname}
\theoremstyle{definition}
\newtheorem{example}[thm]{\protect\examplename}
\theoremstyle{definition}
\newtheorem{problem}[thm]{\protect\problemname}

\usepackage{hyperref}
\usepackage[T2A]{fontenc}

\usepackage{tikz}
\usetikzlibrary{calc,math,intersections}
\usetikzlibrary{shapes.geometric}
\usetikzlibrary{patterns}
\usetikzlibrary{arrows.meta}
\mathcode`l="8000
\begingroup
\makeatletter
\lccode`\~=`\l
\DeclareMathSymbol{\lsb@l}{\mathalpha}{letters}{`l}
\lowercase{\gdef~{\ifnum\the\mathgroup=\m@ne \ell \else \lsb@l \fi}}%
\endgroup

\makeatother

\usepackage{babel}
\providecommand{\corollaryname}{Corollary}
\providecommand{\definitionname}{Definition}
\providecommand{\examplename}{Example}
\providecommand{\lemmaname}{Lemma}
\providecommand{\problemname}{Problem}
\providecommand{\remarkname}{Remark}
\providecommand{\theoremname}{Theorem}

\begin{document}
\global\long\def\d{\mathrm{d}}%

\global\long\def\D{\operatorname{D}}%

\global\long\def\conv{\operatorname{conv}}%

\global\long\def\epi{\operatorname{epi}}%

\global\long\def\diag{\operatorname{diag}}%

\global\long\def\diam{\operatorname{diam}}%

\global\long\def\sr{\operatorname{sr}}%

\global\long\def\vol{\operatorname{vol}}%

\global\long\def\BDV{\operatorname{BDV}}%

\global\long\def\dist{\operatorname{dist}}%

\global\long\def\osc{\operatorname{osc}}%

\global\long\def\rank{\operatorname{rank}}%

\global\long\def\Span{\operatorname{span}}%

\global\long\def\id{\operatorname{id}}%

\global\long\def\aff{\operatorname{aff}}%

\global\long\def\Aff{\operatorname{Aff}}%

\global\long\def\relint{\operatorname{relint}}%

\global\long\def\Int{\operatorname{int}}%

\global\long\def\R{\mathbb{R}}%

\global\long\def\pa{\operatorname{pa}}%

\global\long\def\children{\operatorname{chil}}%

\global\long\def\siblings{\operatorname{sibl}}%

\global\long\def\anc{\operatorname{anc}}%

\global\long\def\desc{\operatorname{desc}}%

\global\long\def\nca{\operatorname{nca}}%

\global\long\def\Simplexe{\mathbb{S}}%

\global\long\def\T{\mathcal{T}}%

\global\long\def\Vertices{\mathcal{\mathcal{V}}}%

\global\long\def\Edges{\mathcal{\mathcal{E}}}%

\global\long\def\cof{\operatorname{cof}}%

\global\long\def\supp{\operatorname{supp}}%

\global\long\def\Sub{\operatorname{Sub}}%

\global\long\def\sub{\operatorname{sub}}%

\global\long\def\Ebis{E_{\operatorname{bis}}}%

\global\long\def\bisE{\operatorname{bis}\mathcal{E}}%

\global\long\def\tria{\operatorname{tria}}%

\global\long\def\Vnew{V_{\operatorname{new}}}%

\global\long\def\vnew{v_{\operatorname{new}}}%

\global\long\def\vnew{v_{\operatorname{new}}}%

\global\long\def\Mid{\operatorname{mid}}%

\global\long\def\RHS{\operatorname{RHS}}%

\global\long\def\Part{\operatorname{Part}}%

\global\long\def\ra{\rightarrow}%

\global\long\def\la{\leftarrow}%

\newcommandx\ar[1][usedefault, addprefix=\global, 1=1]{\xrightarrow{#1}}%

\newcommandx\al[1][usedefault, addprefix=\global, 1=1]{\xleftarrow{#1}}%

\global\long\def\Rstr{\operatorname{Rstr}}%

\global\long\def\Descext{\operatorname{Desc_{\{ext\}}}}%

\global\long\def\smpl{\operatorname{smpl}}%

\global\long\def\mres{\mathbin{\vrule height1.6exdepth0ptwidth0.13ex\vrule height0.13exdepth0ptwidth1.3ex}}%

\title[The Hessian equation for very general data]{The Hessian equation on nonsmooth $k$-convex domains or in the presence
of subsolutions}
\author{J. Lukas Gehring}
\thanks{The author wishes to thank God in heaven for the grace to give such
ideas. Further, he wants to thank the European Union's Horizon 2020
research and innovation programme (Grant agreement No.~891734) for
funding.}
\email{lukas.gehring@uni-jena.de}
\begin{abstract}
It is proved for $k>n/2$ that the Dirichlet problem of the $k$-Hessian
measure on a (nonsmooth and nonuniformly) $k$-convex (also known
as $k$-hyperconvex) domain for any finite Borel measure as right-hand
side and boundary data taken from a $k$-convex and continuous function
on the closure of the domain has a unique solution. Merely \emph{continuous}
boundary data are sufficient for \emph{strictly} $k$-convex domains
(i.e.\ if there exist strong barriers). 

Given a sub- and a supersolution, the existence results are proved
also for \emph{general} domains, \emph{infinite} Borel measures and
\emph{discontinuous} boundary data.
\end{abstract}

\maketitle

\section{Introduction and main results}

The Dirichlet problem of the Hessian operator was studied e.g.\ by
Ivochkina \cite{Ivochkina81Eng,IvochkinaIntegralEng,Ivochkina83stabilityEng},
Caffarelli, Nirenberg and Spruck \cite{CaffarelliNirenbergSpruck3}
who showed existence of a unique smooth solution for smooth data,
and Trudinger and Wang \cite{TrudingerWangHM1,TrudingerWangHM2,TrudingerWangHM3}\cite{WangHessian}.
The latter extended the Hessian operator in a weak sense to a class
of functions which need not be twice differentiable, a subset of the
subharmonic functions and a superset of the locally convex functions,
called them the \emph{$k$-convex} functions, and showed unique solvability
of the $k$-Hessian Dirichlet problem in that class for rather general
domains, boundary values and right-hand sides, see Remark \ref{rem:TW=000020existence}
below. For $k>n/2$ the $k$-convex functions are real-valued, continuous
and easier to deal with than for $k\le n/2$, therefore we limit ourselves
to that case as in \cite{TrudingerWangHM1}. In this paper, we generalize
that theorem in the following respects: 

Firstly, in \cite{TrudingerWangHM1}, the boundary of the domain had
to be $C^{2}$ and uniformly $\Gamma_{k-1}$ (in our terminology).
A corresponding theorem for the Monge--Ampère equation \cite{Hartenstine}
requires merely convex domains. Actually, the history of results of
this type also includes the solution of the Dirichlet problem of the
\emph{complex} Monge--Ampère operator by Bedford and Taylor \cite{BedfordTaylor1976Dirichlet}
for \emph{strictly pseudoconvex }(the analog of uniformly convex)
domains and it was Błocki \cite{BlockiLp} who remarked that hyperconvexity
(the analog of mere convexity) instead of strict pseudoconvexity is
the natural condition for the boundary to solve the Dirichlet problem.
We solve the Dirichlet problem on open bounded $k$-convex sets, i.e.\ sets
which admit a negative-valued continuous $k$-convex function $u_{\Omega}$
such that the sublevels of $u_{\Omega}$ and negative numbers are
compactly included.

Secondly, while in \cite{TrudingerWangHM1} the right-hand side measure
needed to be finite and the sum of an absolutely continuous and a
compactly supported measure, our generalization works for any finite
Borel measure as right-hand side, as the well-established existence
theorem for the Monge--Ampère equation does, see Theorem \ref{thm:Bern}
below.

Thirdly, Guan and Spruck gave examples for the existence of solutions
relying on the existence of strict smooth subsolutions without any
geometrical conditions on $\Omega$ in \cite{GuanSpruck1993,Guan1994DirichletClassFNEE}
and other papers, as it holds generally for viscosity solutions \cite[Thm.~4.1, p.~23]{CrandallIshiiHitoshiUsersGuide}.
Given a subsolution, we want to ask how the other assumptions can
be relaxed. The measure may be infinite and the boundary conditions
can be given by a function on the domain (discontinuous on the closure)
and we say that another function satisfies these boundary conditions
if the difference of the functions decays towards the boundary. We
prove some relative statements about discontinuous boundary data and
infinite Borel measures about the set of solvable right hand-sides:
When the Dirichlet problem is solvable for certain right-hand sides
(e.g.\ with zero boundary conditions), it is also solvable for certain
others (e.g.\ smaller measure and continuous boundary values), see
Theorem \ref{thm:gedaechtniskirche} below.

Let us introduce some standard notation. We denote by $\bar{A}$ the
closure of a set $A\subset\mathbb{R}^{n}$. For functions $u,w$ on
a domain of definition $\Omega$, $\{u\le w\}$ is an abbreviation
for the set $\left\{ x\in\Omega~\middle|~u(x)\le w(x)\right\} $ and
the like. Given $\Omega\subset\mathbb{R}^{n}$ and a subset of a normed
space $A$, $C_{0}(\Omega,A)$ and $C_{0}(\Omega)$ is the set of
continuous functions $u$ from $\Omega$ to $A$ and to $\mathbb{R}$,
respectively, such that $\{\left|u\right|\le\varepsilon\}\subset\Omega$
is compact for all $\varepsilon>0$. If we write something like $u\in C(K)$,
$u$ is often defined on a superset of $K$ and we mean that the restriction
$u|_{K}$ is continuous. We say that a sequence $u_{N}\in C(\Omega_{N})$
converges in $C_{\mathrm{loc}}(\Omega)$ if for any compact $K\subset\Omega$
there exists $N_{0}>0$ such that $K\subset\Omega_{N}$ for all $N\ge N_{0}$
and $u_{N}|_{K}$ converges in $C(K)$. The expression $A\subset\subset B$
means that $\bar{A}\subset B$ and $\bar{A}$ is compact. 

We will sometimes abbreviate the Euclidean distance from a set $A$
to a point $b$ by 
\[
\dist_{A}(b):=\dist(A,b)=\inf_{a\in A}\left|a-b\right|.
\]

Let us briefly define the $k$-Hessian operator, $k$-convex functions
and uniformly $k$-convex sets. We try to abide by the notation of
Trudinger and Wang. For $k=1,\dots,n$, and $u\in C^{2}(\Omega)$,
the $k$-Hessian operator $F_{k}$ is defined by
\[
F_{k}[u]:=S_{k}(\lambda(D^{2}u)),
\]
where $\lambda=(\lambda_{1},\dots,\lambda_{n})$ are the eigenvalues
of the Hessian matrix and $S_{k}$ is the $k$-th elementary symmetric
function on $\mathbb{R}^{n}$, given by 
\[
S_{k}(\lambda):=\sum_{i_{1}<\dots<i_{k}}\lambda_{i_{1}}\cdots\lambda_{i_{k}}.
\]
We denote the associated \emph{Gårding cone} (see \cite{Garding}\cite[(2.5)]{TrudingerWangHM2})
by 

\[
\Gamma_{k}:=\Gamma_{k}^{n}:=\left\{ \lambda\in\mathbb{R}^{n}~\middle|~\forall\eta_{i}\ge\lambda_{i}\quad S_{k}(\eta)>0\right\} .
\]
A \emph{$k$-convex function $u:\Omega\to\left[-\infty,\infty\right[$}
is an upper semi-continuous function with the property that if $\phi$
is a quadratic polynomial such that $u-\phi$ has a local maximum
at $x\in\Omega$, then $F_{k}[\phi]\geq0$. In particular, the $1$-convex
functions are the subharmonic functions, the $n$-convex functions
are the locally convex functions. Let $\Phi^{k}(\Omega)$ be the set
of $k$-convex functions on $\Omega$. A function $u\in C^{2}(\Omega)$
is $k$-convex if the eigenvalues of $D^{2}u(x)$ lie in $\bar{\Gamma}_{k}$
for any $x\in\Omega$. For $k>n/2$, $\Phi^{k}(\Omega)\subset C(\Omega)$
\cite[Thm.~2.7]{TrudingerWangHM2}. We say that a sequence of (nonnegative)
measures $\nu_{j}$ on $\Omega$ converges to $\nu$ weakly{*} and
write $\nu_{j}\rightharpoonup^{*}\nu$ if $\int_{\Omega}\phi\,\d\nu_{j}\to\int_{\Omega}\phi\,\d\nu$
for any compactly supported continuous function $\phi\in C_{c}(\Omega)$.
Like the Aleksandrov formulation of the Monge--Ampère operator, the
$k$-Hessian operator can be extended to the $k$-convex functions
as shown by Trudinger and Wang in \cite[Thm.~1.1]{TrudingerWangHM2}
and \cite[Thm.~8.1]{WangHessian}: For any $u\in\Phi^{k}(\Omega)$,
there exists a locally finite Borel measure $\mu_{k}[u]$ on $\Omega$,
the\emph{ $k$-Hessian measure generated by} $u$, such that $\mu_{k}[u]=F_{k}[u]\,dx$
if $u\in C^{2}(\Omega)$ and the following property holds: \emph{If
$u_{m}$ is a sequence with
\begin{align}
\begin{array}{rl}
\Phi^{k}(\Omega)\ni u_{m}\negthickspace\negthickspace\negmedspace & \to u\in\Phi^{k}(\Omega)\quad\text{loc. in measure or a.e.,}\\
\text{then}\qquad\mu_{k}[u_{m}]\negthickspace\negthickspace\negmedspace & \rightharpoonup^{*}\mu_{k}[u].
\end{array}\label{eq:weak*=000020conti}
\end{align}
}As the most prominent examples, $\mu_{1}$ is the Laplace operator,
$\mu_{n}$ is the Monge--Ampère measure. 

The following definition is supplemented by many equivalent ones in
the author's upcoming paper \cite{Gehring2025kconvex}.
\begin{defn}[$k$-convex and strictly $k$-convex sets, negative exhaustion function,
strong barrier]
We say that an open bounded set is \emph{$k$-convex} (also known
as $k$-hyperconvex) if there exists a $k$-convex \emph{negative
exhaustion function} (also known as level function), that is a function
$u_{\Omega}\in\Phi^{k}(\Omega)\cap C_{0}(\Omega,\left]-\infty,0\right[)$.
A $k$-convex set is \emph{strictly $k$-convex} (also known as $B_{k}$-regular)\emph{
if} for any boundary point $y_{0}\in\partial\Omega$ there exists
a continuous $k$-convex \emph{strong barrier }function, i.e.\ a
function $u_{\Omega,y_{0}}\in\Phi^{k}(\Omega)\cap C(\bar{\Omega})$
with a strict global maximum in $y_{0}$. (The supposition that the
set is $k$-convex is redundant, see \cite{Gehring2025kconvex}.)
\end{defn}

These definitions originate from the notions of a \emph{regular} domain
for the Laplace operator and \emph{pseudoconvex, $B$-regular,} and
\emph{hyperconvex} domains for plurisubharmonic functions, i.e.\ for
the complex Monge--Ampère operator, see e.g.\ \r{A}hag, Czy\.{z}
and Hed \cite{AhagCzyzHed} for more information. In contrast to \cite{WanWangLelong},
we prefer the designation \emph{$k$-convex }set, because the connected
$n$-convex sets are exactly the open convex sets (see \cite{Gehring2025kconvex}
for a proof).

Some authors define a smooth $k$-convex subset of $\mathbb{R}^{n}$
as a smooth set with a $k$-convex exhaustion function, while others
(e.g.\ \cite{TrudingerWangHM1,TrudingerWangHM2,TrudingerWangHM3,WangHessian})
as a set with principal curvatures lying in $\bar{\Gamma}_{k}^{n-1}$
which causes ambiguity of the meaning of $k$, because a $k$-convex
exhaustion function corresponds to principal curvatures in $\bar{\Gamma}_{k-1}^{n-1}$
(see Lemma \ref{thm:k-1=000020konvex}). While both options can be
justified well, we follow the former terminology like \cite{AhagCzyzHed,WanWangLelong}. 

We say that an open set $\Omega$ of class $C^{2}$ has a \emph{uniform
$\Gamma_{k}$-boundary} if the principal curvatures (see Definition
\ref{def:principals}) lie in $\Gamma_{k}$ and $S_{k}(\kappa_{1},\dots,\kappa_{n})\ge c>0$
at any point of the boundary. In \cite{TrudingerWangHM1,TrudingerWangHM2,TrudingerWangHM3,WangHessian}
such a set was called $k$-convex. Such a contrasting name should
avoid confusion with our definition of a $k$-convex set, but in \cite{Gehring2025kconvex}
we show that a $C^{2}$-bounded set is $k$-convex if and only if
its principal curvatures lie in $\bar{\Gamma}_{k-1}$.

The first principal result is the following.
\begin{thm}
\label{thm:Bern}Let $k>n/2$, $\Omega\subset\mathbb{R}^{n}$ be open,
bounded, and $k$-convex. Let $u_{\partial}\in\Phi^{k}(\Omega)\cap C(\bar{\Omega})$
and $\nu$ be a finite Borel measure on $\Omega$. Then the Dirichlet
problem
\begin{equation}
\begin{cases}
\mu_{k}[u]=\nu & \text{in }\Omega\\
u=u_{\partial} & \text{on }\partial\Omega
\end{cases}\label{eq:DP_alt}
\end{equation}
has a unique solution $u\in\Phi^{k}(\Omega)\cap C(\bar{\Omega})$.

The same holds true if $\Omega$ is strictly $k$-convex and $u_{\partial}\in C(\partial\Omega)$
instead of $u_{\partial}\in\Phi^{k}(\Omega)\cap C(\bar{\Omega})$. 
\end{thm}

Besides, strict $k$-convexity is necessary to solve all Dirichlet
problems with continuous boundary values, see Lemma \ref{lem:k-convex=000020extension}.
\begin{rem}[Related prior work]
\label{rem:TW=000020existence}Trudinger and Wang \cite[Thm.~1.2]{TrudingerWangHM1}
showed for $k>n/2$ that the Dirichlet problem has a unique solution
if $\Omega$ is open, bounded, has a uniform $\Gamma_{k{-}1}$-boundary,
$\nu$ is the sum of a measure with $L^{1}$-density and a compactly
supported finite Borel measure, and $u_{\partial}\in C(\bar{\Omega})$.
Hartenstine \cite{Hartenstine} proved Theorem \ref{thm:Bern} for
the case $k=n$. For the complex Monge--Ampère equation, Błocki \cite{BlockiLp}
showed that the Dirichlet problem of the complex Monge--Ampère operator
is uniquely solvable on hyperconvex domains (the analog of $k$-convex
sets) for continuous right-hand sides and boundary data as ours.
\end{rem}

\begin{rem}
Froese, Oberman and Salvador presented a numerical scheme to solve
the 2-Hessian equation in 3 dimensions \cite{FroeseObermanSalvador}.
In their Remark 3.4 they write ``The assumption of the existence
of a continuous viscosity solution {[}...{]} is restrictive since
the existence result of Theorem 2.4 requires smooth data, which is
not the case for the examples considered here. {[}...{]} However,
a precise well-posedness result for the (weak) Dirichlet problem is
not presently available.'' In their Section 4, they solve the 2-Hessian
equation on a cube. Our Theorem \ref{thm:Bern} lays a theoretical
fundament that at least a continuous weak solution exists.
\end{rem}

In our second main result, we try to collect what can be said about
discontinuous boundary data and infinite right-hand side measures.
It should be understood as a couple of tools to build solvable Dirichlet
problems. To embrace discontinuous boundary data, the Dirichlet problem
reads
\begin{equation}
\begin{cases}
\mu_{k}[u]=\nu & \text{in }\Omega\\
u-u_{\partial}\in C_{0}(\Omega)
\end{cases}\label{eq:DP=000020discontinuous}
\end{equation}
with $u_{\partial}\in\Phi^{k}(\Omega)$ here, generalizing (\ref{eq:DP_alt}).
\begin{thm}
\label{thm:gedaechtniskirche}Let $k>n/2$, $\Omega\subset\mathbb{R}^{n}$
be open. The set of \emph{solvable right-hand sides}
\[
\RHS:=\left\{ (\mu_{k}[u],u_{\partial})~\middle|~u\in\Phi^{k}(\Omega),u_{\partial}\in C(\Omega),u{-}u_{\partial}\in C_{0}(\Omega)\right\} 
\]
has the following properties:
\begin{enumerate}
\item \label{enu:uniqueness}(uniqueness) For any $(\nu,u_{\partial})\in\RHS$,
there exists only one solution $u\in\Phi^{k}$ with $\mu_{k}[u]=\nu$
and $u-u_{\partial}\in C_{0}(\Omega)$.
\item \label{enu:Klassiker}(equals Theorem \ref{thm:Bern}) If $\Omega$
is bounded and $k$-convex, then $(\nu,u_{\partial})\in\RHS$ for
all finite Borel measures $\nu$ on $\Omega$ and all $u_{\partial}\in\Phi^{k}(\Omega)\cap C(\bar{\Omega})$.
\item \label{enu:Zwischenwertsatz}(subsolution and supersolution) If $(\nu_{1},u_{\partial}),(\nu_{3},u_{\partial})\in\RHS$,
then $(\nu_{2},u_{\partial})\in\RHS$ for all $\nu_{1}\le\nu_{2}\le\nu_{3}$.
\item \label{enu:Addition}(addition) If $(\nu_{1},u_{1}),(\nu_{2},u_{2})\in\RHS$
and there exists $(\nu_{3},u_{1}{+}u_{2})\in\RHS$ with $\nu_{3}\le\nu_{1}+\nu_{2}$,
then also $(\nu_{1}+\nu_{2},u_{1}+u_{2})\in\RHS$.
\end{enumerate}
\end{thm}

\begin{rem}
If the boundary values satisfy $u_{\partial}\in\Phi^{k}(\Omega)\cap C(\bar{\Omega})$,
the ``supersolution'' $(\nu_{1},u_{\partial}):=(0,u_{\partial})\in\RHS$
in (\ref{enu:Zwischenwertsatz}) is readily given by Lemma (\ref{lem:k-convex=000020envelope}).

The statement (\ref{enu:Zwischenwertsatz}) for $k=n$ and discontinuous
boundary data (even in a weaker sense) was somehow hinted in \cite[Thm.~2.1, Cor.~2.2]{trudinger_wang},
but in the depended Lemmas \cite[Lem.~2.4, Lem.~3.3]{trudinger_wang}
continuous boundary values are assumed. The author is not sure if
they had discontinuous boundary data in mind.\medskip{}
The remaining article is structured as follows: In the next section,
we prove Theorem \ref{thm:Bern} and in Section \ref{sec:kreisbeispiel},
we prove Theorem \ref{thm:gedaechtniskirche} and give examples of
a convex functions with discontinuous boundary values, soluble, and
an insoluble Dirichlet problem. In the concluding Section \ref{sec:Open-questions},
we list some open questions and in the Appendix, we cite some theorems
from others, on which our results are founded.
\end{rem}

\section{\protect\label{sec:Dirichlet=000020problem}The Dirichlet problem
on $k$-convex domains}

In this section, we prove Theorem \ref{thm:Bern}. Here, $\Omega$
is always a bounded open subset of $\mathbb{R}^{n}$. The proofs are
prepared by the following lemmas (and the cited results in Appendix
\ref{sec:Theorems-from-others}).

According to \cite{CaffarelliNirenbergSpruck3}, $F_{k}^{1/k}$ is
increasing, homogeneous and concave on $\bar{\Gamma}_{k}$, which
implies the superadditivity

\begin{align}
\mu_{k}[u{+}w] & \ge\mu_{k}[u]+\mu_{k}[w]\label{eq:superadditivity}
\end{align}
for $u,w\in\Phi^{k}$.
\begin{lem}[{Comparison principle, cf.~\cite[Thm.~3.1]{TrudingerWangHM1}, \cite[Lem.~8.2]{WangHessian}}]
\emph{\label{lem:comparison=000020principle}}Let $k>n/2$ and $u,w\in\Phi^{k}(\Omega)$.
If
\begin{align}
 &  & \{u\ge w+\varepsilon\} & \subset\subset\Omega & \text{for all }\varepsilon>0\nonumber \\
 & \text{and} & \mu_{k}[u] & \ge\mu_{k}[w] & \text{in }\Omega,\label{eq:comparison=000020principle}\\
 & \text{then also} & u & \le w & \text{in }\Omega.\nonumber 
\end{align}
\end{lem}

\begin{proof}
Assume to the contrary that the first and the second condition are
valid, but $u(x_{0})>w(x_{0})$. Let us construct a function $u'$
such that $u'(x_{0})>w(x_{0})$, $u'<w$ in a neighborhood of the
boundary and $\mu_{k}[u']-c\,dx\ge\mu_{k}[w]$ for a small $c>0$.
Subtraction of $(u(x_{0})-w(x_{0}))/2$ from $u$ brings the first
two properties. A subsequent addition of $\varepsilon\left|x\right|^{2}$,
(i.e., let $u'(x):=u(x)+(w(x_{0})-u(x_{0}))/2+\varepsilon\left|x\right|^{2}$)
conserves these properties for sufficiently small $\varepsilon>0$
and leads to $\mu_{k}[u']\geq\mu_{k}[w]+\binom{n}{k}\varepsilon^{k}dx$
by the superadditivity. Therefore $\{u'>w\}$ is open, nonempty and
compactly included in $\Omega$. The monotonicity Lemma \ref{lem:monotonicity}
implies that $\mu_{k}[w](\{u'>w\})\geq\mu_{k}[u'](\{u'>w\})$, which
is a contradiction.
\end{proof}
The comparison principle implies uniqueness for the Dirichlet problem.

For Lemma \ref{thm:k-1=000020konvex} we need the following results
about mollification of $k$-convex functions from \cite{TrudingerWangHM2}.
\begin{lem}[{Mollification of $k$-convex functions, equals \cite[Lem.~2.3]{TrudingerWangHM2}}]
\label{lem:Mollifier}For a spherically symmetric mollifier $\rho\in C_{c}^{\infty}\left(\mathbb{R}^{n}\right)$
satisfying $\rho>0$ for $\left|x\right|<1$, $\rho=0$ for $\left|x\right|\geq1$,
and $\int_{B}\rho\,\d x=1$, let $\rho_{h}(x):=h^{-n}\rho(x/h)$.
Let $u\in\Phi^{k}(\Omega).$ Then the convolution $u*\rho_{h}\in C^{\infty}(\Omega_{h})\cap\Phi^{k}(\Omega_{h})$,
where $\Omega_{h}:=\left\{ x\in\Omega~\middle|~\dist_{\partial\Omega}(x)>h\right\} $.
Moreover, as $h\searrow0$, the sequence $u*\rho_{h}\searrow u$ in
$L_{\mathrm{loc}}^{1}$. For $k>n/2$ even $u_{h}\searrow u$ locally
uniformly.
\end{lem}

One of the main ideas to prove existence in Theorem \ref{thm:Bern}
is to reduce it to Trudinger's and Wang's $\Phi^{k}$ existence theorem
(see Remark \ref{rem:TW=000020existence}) which is mainly done by
the following Lemma.
\begin{lem}
\label{thm:k-1=000020konvex}For a $k$-convex set $\Omega$, there
exists a sequence $\mathcal{U}_{1}\subset\subset\mathcal{U}_{2}\subset\subset\cdots$
of uniformly $\bar{\Gamma}_{k-1}$-bounded smooth subsets with $\bigcup_{i\in\mathbb{N}}\mathcal{U}_{i}=\Omega$.
\end{lem}

The proof relies on Caffarelli's, Nirenberg's, and Spruck's result
Lemma \ref{rem:uniform=000020domain}
\begin{proof}
Let $\delta_{1}\ge\delta_{2}\ge\cdots>0$ be a decreasing sequence
such that $|x{-}y|<\delta_{N}$ implies $|u_{\Omega}(x)-u_{\Omega}(y)|<1/N$
for all $x,y\in\bar{\Omega}$. Let $\rho_{h}$ be as in Lemma \ref{lem:Mollifier},
$1/N<c_{N}<1/(N{-}1)$ to be defined later, and
\begin{align*}
u_{N} & :\bar{\Omega}_{\delta_{N}}:=\left\{ \dist_{\mathbb{R}^{n}\setminus\Omega}\ge\delta_{N}\right\} \to\mathbb{R}\\
u_{N} & :=u_{\Omega}\ast\rho_{h}+c_{N}+\left|\cdot\right|^{2}/N.
\end{align*}
This is a smooth $k$-convex function on $\Omega_{\delta_{N}}$ according
to Lemma \ref{lem:Mollifier}. Due to the term $\left|\cdot\right|^{2}/N$,
$F_{k}[u_{N}]$ has a positive lower bound. Let us bound $u_{N}$
from below and above. Let $x\in\bar{\Omega}_{\delta_{N}}$. By the
definition of $\delta_{N}$, the value of $u_{\Omega}$ oscillates
in $B_{\delta_{N}}(x)\subset\Omega$ by at most $1/N$ from $u_{\Omega}(x)$,
and by Lemma \ref{lem:Mollifier}, the convolution term satisfies
\begin{equation}
u_{\Omega}\leq u_{\Omega}*\rho_{h}\leq u_{\Omega}+1/N.\label{eq:Mollification=000020estimate}
\end{equation}
Thus, 
\[
u_{\Omega}+1/N<u_{N}\leq u_{\Omega}+2/(N{-}1)+\left|\cdot\right|^{2}/N.
\]
On $\partial\Omega_{\delta_{N}}$, the definition of $\delta_{N}$
implies $u_{N}(x)>u_{\Omega}(x)+1/N\ge0$. Is a compactly contained
sublevel of the smooth function $u_{N}$ automatically smooth? The
implicit function theorem implies it only for noncritical boundary
points. Indeed, if $x\in\Omega_{\delta_{N}}$ is a saddle point of
$u_{N}$, the set $\{u_{N}=u_{N}(x)\}$ cannot be smooth (and in contrast
to convex functions, $k$-convex functions may have saddle points),
so we need a boundary of noncritical points. Sard's theorem asserts
that the set of critical values, $\{u_{N}(x)\allowbreak~|\allowbreak~\exists x\in\Omega\allowbreak\ \nabla u_{N}(x)=0\}$,
is a null set in $\mathbb{R}$. Therefore, we can always choose a
noncritical value $1/N<c_{N}<1/(N{-}1)$ such that $\nabla u_{N}(x)\neq0$
wherever $u_{N}(x)=0$ and hence
\[
\mathcal{U}_{N}:=\{u_{N}<0\}\subset\subset\Omega_{\delta_{N}}
\]
has a smooth boundary. By Lemma \ref{rem:uniform=000020domain}, the
lower bound for $F_{k}[u_{N}]$ implies that $\mathcal{U}_{N}$ has
a uniform $\Gamma_{k{-}1}$-boundary. The upper bound $u_{N}\leq u_{\Omega}+2/(N{-}1)+\left|\cdot\right|^{2}/N$
and $u_{N}<0$ in $\Omega$ imply $\cup_{N\in\mathbb{N}}\mathcal{U}_{N}=\Omega$.
The monotonicity $u_{N}\ast\rho_{\delta_{N}}\ge u_{N}\ast\rho_{\delta_{N+1}}$
in Lemma \ref{lem:Mollifier} implies that $u_{1}>u_{2}>\cdots$ and
thus the nesting property $\mathcal{U}_{1}\subset\subset\cdots$.
\end{proof}
\begin{rem}
In \cite{Gehring2025kconvex}, we drop the assumption of the boundedness
and show the converse: The limit of a sequence $\mathcal{U}_{1}\subset\cdots$
of $\bar{\Gamma}_{k-1}$-bounded sets is $k$-convex.
\end{rem}

Labutin showed the existence of a $k$-\emph{convex envelope} of the
trace of a $k$-convex function on the boundary for any open bounded
domain:
\begin{lem}[{$k$-convex envelope, see \cite[Thm.~6.3 and text around]{Labutin}}]
\label{lem:k-convex=000020envelope}Let $k>n/2$ and $u\in\Phi^{k}(\Omega)\cap C(\bar{\Omega})$.
Then there exists the $k$\emph{-convex envelope} $(u|_{\partial\Omega})_{*}$,
i.e.\ a unique solution of
\[
\begin{cases}
\mu_{k}[(u|_{\partial\Omega})_{*}]=0 & \text{in }\Omega\\
(u|_{\partial\Omega})_{*}=u & \text{on }\partial\Omega,
\end{cases}
\]
which is given by $(u|_{\partial\Omega})_{*}=\sup\left\{ w\in\Phi^{k}(\Omega)\:\middle|\:w|_{\partial\Omega}\leq u|_{\partial\Omega}\right\} $.
\end{lem}

Let us use the $k$-convex envelope to construct a Green's function
for $\mu_{k}$ and $k$-convex sets $\Omega$.
\begin{lem}[Green's function]
\label{lem:Green's=000020function}Let $k>n/2$ and $\Omega$ be
$k$-convex, $\delta_{x}$ be the Dirac measure in $x\in\Omega$.
There exists a unique solution $G_{k,\Omega,x}\in C_{0}(\Omega)\cap\Phi^{k}(\Omega)$
of 
\[
\mu_{k}[G_{k,\Omega,x}]=\delta_{x}.
\]
\end{lem}

\begin{proof}
Let $u_{\Omega}\in\Phi^{k}(\Omega)\cap C_{0}(\Omega,\left]-\infty,0\right[)$
be a negative exhaustion function for $\Omega$ and $w:=(u_{\Omega}|_{\partial\Omega\cup\{x\}})_{*}$
(apply Lemma \ref{lem:k-convex=000020envelope} to $\Omega\setminus\{x\}$)
such that $\mu_{k}[w](\Omega\setminus\{x\})=0$. Additionally $w(x)=u_{\Omega}(x)<0$,
which implies that $\mu_{k}[w](\{x\})>0$, because the only function
$w\in C_{0}(\Omega)\cap\Phi^{k}(\Omega)$ with $\mu_{k}[w]=0$ is
the zero function (according to the comparison principle). On the
other hand, $\mu_{k}[w]$ is locally finite and therefore $\mu_{k}[w](\{x\})<\infty$.
Thus, $w/\mu_{k}[w](\{x\})^{1/k}$ is the wanted Green's function.
\end{proof}
Since $G_{k,\Omega,x}$ is subharmonic and $G_{k,\Omega,x}(x)<0$,
it is negative on $\Omega$ by the strong maximum principle for subharmonic
functions \cite[Thm.~2.2, p.~15]{GilbargTrudinger2001}. By the comparison
principle (\ref{eq:comparison=000020principle}), the Green's function
is monotone in $\Omega$, i.e.
\begin{equation}
x\in\Omega_{1}\subset\Omega_{2}\quad\text{implies}\quad G_{k,\Omega_{1},x}\geq G_{k,\Omega_{2},x}|_{\Omega_{1}}.\label{eq:Green's=000020monotonicity}
\end{equation}
The next Lemma and Corollary use the notions of the $k$-convex envelope
and the Green's function to derive an Aleksandrov estimate.
\begin{lem}[Pointwise comparison principle]
\label{lem:Pointwise=000020comparison=000020principle}Let $k>n/2$,
$x\in\Omega$, and let two functions $u,w\in\Phi^{k}(\Omega)\cap C(\bar{\Omega})$
satisfy
\begin{align}
 &  & u & \le w & \text{on }\partial\Omega\nonumber \\
 & \text{and} & \mu_{k}[u](\{x\}) & \ge\mu_{k}[w](\Omega). & \text{in }\Omega.\label{eq:weighted=000020comparison=000020principle}\\
 & \text{Then also} & u(x) & \le w(x).\nonumber 
\end{align}
\end{lem}

\begin{proof}
Let $\check{w}\ge w$ be the $k$-convex envelope of $\check{w}=w$
on $\partial\Omega$ and $\check{w}(x)=w(x)$ (use Lemma \ref{lem:k-convex=000020envelope}
for $\Omega\setminus\{x\}$). The monotonicity (Lemma \ref{lem:monotonicity})
says that 
\[
\mu_{k}[\check{w}]=\mu_{k}[\check{w}](\Omega)\delta_{x}\le\mu_{k}[w](\Omega)\delta_{x}\le\mu_{k}[u].
\]
Therefore, the comparision principle (\ref{eq:comparison=000020principle})
applied to $u$ and $\check{w}$ implies that $u(x)\leq\check{w}(x)=w(x)$.
\end{proof}
\begin{cor}[Aleksandrov estimate]
\label{cor:AMP=000020mit=000020Randdaten}Let $k>n/2$, $\Omega$
be $k$-convex, and $u\in\Phi^{k}(\Omega)\cap C(\overline{\Omega})$.
Then 
\[
u-\left(u|_{\partial\Omega}\right)_{*}\geq\mu_{k}[u](\Omega)^{1/k}G_{k,\Omega,x}.
\]
\end{cor}

\begin{proof}
Let $\tilde{u}:=\left(u|_{\partial\Omega}\right)_{*}+\mu_{k}[u](\Omega)^{1/k}G_{k,\Omega,x}$.
The superadditivity (\ref{eq:superadditivity}) says that
\begin{align*}
\mu_{k}[\tilde{u}] & \ge\mu_{k}[u](\Omega)\delta_{x}
\end{align*}
and the pointwise comparison principle (Lemma \ref{lem:Pointwise=000020comparison=000020principle})
implies that 
\[
\left(u|_{\partial\Omega}\right)_{*}(x)+G_{k,\Omega,x}(x)\mu_{k}[u](\Omega)^{1/k}=\tilde{u}(x)\leq u(x).
\]
\end{proof}
\begin{proof}[Proof of Theorem \ref{thm:Bern}]
Compared to the established existence proofs of the Monge--Ampère
case (cf.\ \cite{Gutierrez,Hartenstine,trudinger_wang,Figalli2017}),
the proof here works without the fact that the peak of the Green's
function (i.e.\ the factor $G_{k,\Omega,x}(x)$ in the Aleksandrov
estimate (Corollary \ref{cor:AMP=000020mit=000020Randdaten})) decays
at the boundary. The reader is invited to read the proof assuming
$u_{\partial}=0$ at first, to comprehend the idea clearlier.

\emph{The uniqueness follows from the comparison principle (\ref{eq:comparison=000020principle}).}

\emph{At first, let us assume that $\nu$ is compactly supported in
$\Omega$.} By Lemma \ref{thm:k-1=000020konvex}, there exists a sequence
$\mathcal{U}_{1}\subset\subset\mathcal{U}_{2}\subset\subset\cdots$
of domains with $\Gamma_{k{-}1}$-boundary and $\bigcup_{N\in\mathbb{N}}\mathcal{U}_{N}=\Omega$.
By Lemma \ref{lem:k-convex=000020envelope}, we can replace $u_{\partial}$
by $(u_{\partial}|_{\partial\Omega})_{*}$, keeping the boundary values,
i.e.\ in the following we assume that $\mu_{k}[u_{\partial}]=0$.
Let $N_{0}$ be sufficiently large such that $\supp\nu\subset\mathcal{U}_{N_{0}}.$
Now, by Trudinger's and Wang's existence result (see Remark \ref{rem:TW=000020existence}),
let $u_{N}\in C(\mathcal{U}_{N})$ be the solution of
\begin{align*}
\begin{cases}
\mu_{k}[u_{N}]=\nu & \text{in }\mathcal{U}_{N}\\
u_{N}=u_{\partial} & \text{on }\partial\mathcal{U}_{N}
\end{cases}
\end{align*}
for $N\geq N_{0}$. The aim is to show that $u_{N}$ is a Cauchy sequence
in $C_{\mathrm{loc}}(\Omega)$ (i.e., a Cauchy sequence when restricted
to any compact subset) and converges to one (and the only) solution
of (\ref{eq:DP_alt}). By the Aleksandrov estimate (Corollary \ref{cor:AMP=000020mit=000020Randdaten})
and the monotonicity of $G_{k,\Omega,x}$ (\ref{eq:Green's=000020monotonicity}),
$x\in\mathcal{U}_{N}$ satisfies 
\[
u_{N}(x)-u_{\partial}(x)\ge G_{k,\mathcal{U}_{N},x}(x)\mu_{k}[u_{N}](\mathcal{U}_{N})^{\frac{1}{k}}\ge G_{k,B_{\diam(\Omega)}(x),x}(x)\nu(\Omega)^{\frac{1}{k}}=:-C_{k,\Omega,\nu}.
\]
Choose another constant $c>0$ sufficiently large such that $cu_{\Omega}\le-C_{k,\Omega,\nu}\le u_{N}-u_{\partial}$
on $\partial\mathcal{U}_{N_{0}}$ depending on $N_{0}$ but not on
$N$, choose $\varepsilon>0$, and $N_{1}>N_{0}$ sufficiently large
that $cu_{\Omega}\geq-\varepsilon$ outside $\mathcal{U}_{N_{1}}$.
Now we apply the comparison principle (\ref{eq:comparison=000020principle})
to $u_{\partial}+cu_{\Omega}$, $u_{N}$, and $\mathcal{U}_{N}\setminus\bar{\mathcal{U}}_{N_{0}}$
for $N\ge N_{1}$: On $\partial\mathcal{U}_{N}$, $u_{\partial}+cu_{\Omega}\le u_{\partial}=u_{N}$.
On $\partial\mathcal{U}_{N_{0}}$, $u_{\partial}+cu_{\Omega}\le u_{\partial}-C_{k,\Omega,\nu}\leq u_{N}$.
Furthermore, $\mu_{k}[u_{\partial}+cu_{\Omega}]\ge0=\mu_{k}[u_{N}]$
in $\mathcal{U}_{N}\setminus\bar{\mathcal{U}}_{N_{0}}$. Consequently,
$u_{\partial}+cu_{\Omega}\leq u_{N}$ in $\mathcal{U}_{N}\setminus\bar{\mathcal{U}}_{N_{0}}$
and particularly
\begin{equation}
u_{\partial}-\varepsilon\le u_{\partial}+cu_{\Omega}\leq u_{N}\quad\text{in }\mathcal{U}_{N}\setminus\mathcal{U}_{N_{1}}.\label{eq:partial=000020nahe=000020int}
\end{equation}
Now we apply the comparison principle to $u_{N}$, $u_{N_{1}}$, $u_{N}+\varepsilon$,
and $\mathcal{U}_{N_{1}}$: $u_{N}\leq u_{N_{1}}=u_{\partial}\leq u_{N}+\varepsilon$
on $\partial\mathcal{U}_{N_{1}}$ and $\mu_{k}[u_{N}]=\mu_{k}[u_{N_{1}}]$
imply 
\[
u_{N}\le u_{N_{1}}\leq u_{N}+\varepsilon\quad\text{in }\mathcal{U}_{N_{1}}.
\]
Therefore $u_{N}$ is a Cauchy sequence in $C(\mathcal{U}_{N_{1}})$.
Let $u$ be the limit of this locally uniformly convergent sequence.
Since (\ref{eq:partial=000020nahe=000020int}) and the comparison
principle imply $u_{\partial}-\varepsilon\le u_{N}\le u_{\partial}$
outside $\mathcal{U}_{N_{1}}$, $u$ is continuously extendable by
$u_{\partial}$ on $\partial\Omega$. The weak{*} continuity (\ref{eq:weak*=000020conti})
concludes that $u$ solves (\ref{eq:DP_alt}).

\emph{Now we drop the assumption that $\nu$ is compactly supported
in $\Omega$.} Let $K_{1}\subset K_{2}\subset\cdots$ be a sequence
of compact sets with $\bigcup_{i=1}^{\infty}K_{i}=\Omega$ and $u_{j}$
be the solution of 
\[
\begin{cases}
\mu_{k}[u_{j}]=\nu(\cdot\cap K_{j}) & \text{in }\Omega\\
u_{j}=u_{\partial} & \text{on }\partial\Omega,
\end{cases}
\]
and $u_{j}^{i}$ be the solution of 
\[
\begin{cases}
\mu_{k}[u_{j}^{i}]=\nu(\cdot\cap K_{j}\setminus K_{i}) & \text{in }\Omega\\
u_{j}^{i}=0 & \text{on }\partial\Omega.
\end{cases}
\]
Since $\nu(\Omega)<\infty$, the Aleksandrov estimate Corollary \ref{cor:AMP=000020mit=000020Randdaten}
shows that for any $\varepsilon>0$ there exists an $i_{0}\in\mathbb{N}$
such that $\nu(\Omega\setminus K_{i_{0}})$ is sufficiently small
such that $u_{j}^{i}\geq-\varepsilon$ for any $j\geq i\geq i_{0}$.
The superadditivity and the comparison principle imply that $u_{i}+u_{j}^{i}\leq u_{j}\leq u_{i},$
hence $0\leq u_{i}-u_{j}\leq\varepsilon.$ Therefore, $u_{j}$ is
a Cauchy sequence in $C(\bar{\Omega})$ and converges to a limit $u$
with boundary values $u_{\partial}$ uniformly.

\emph{The statement about strict $k$-convex domains follows immediately
from the next Lemma \ref{lem:k-convex=000020extension}.}
\end{proof}
Alternatively, the first step of the proof of the existence (when
the right-hand side measure has compact support) could be proven without
Lemma \ref{thm:k-1=000020konvex} in the following way: At first,
Theorem \ref{thm:gedaechtniskirche}(\ref{enu:Zwischenwertsatz})
is proved for \emph{discrete measures} $\nu$, i.e.\ measures which
are linear combinations of Dirac measures. Since the restriction of
such a measure to any open set is compactly supported, at the point
where we use Theorem \ref{thm:Bern} in the proof of Theorem \ref{thm:gedaechtniskirche}(\ref{enu:Zwischenwertsatz}),
the existence theorem of Trudinger and Wang (see Remark \ref{rem:TW=000020existence})
implies the desired statement. Then Theorem \ref{thm:gedaechtniskirche}(\ref{enu:Zwischenwertsatz})
is used to solve the equation for discrete measures (an upper barrier
is given by the $k$-convex envelope $(u_{\partial}|_{\partial\Omega})_{*}$,
a lower barrier by $u_{\partial}+C(\sum_{i=1}^{N}G_{k,\Omega,x_{i}})$
with $x_{i}$ being the points where the measure is concentrated).
Furthermore, for a compact $K\subset\Omega$, a constant $0<C<\infty$
and $u_{\partial}\in\Phi^{k}(\Omega)\cap C(\bar{\Omega})$, the set
\[
\left\{ u\in\Phi^{k}(\Omega)\cap C(\bar{\Omega})~\middle|~\mu_{k}[u](\Omega\setminus K)=0,\mu_{k}[u](K)\le C,u=u_{\partial}\right\} 
\]
is uniformly bounded by the Aleksandrov estimate, and therefore compact
in $C(K)$ by Corollary \ref{cor:Grenzwerts=0000E4tze} and also in
$C(\bar{\Omega})$ by simple arguments. Hence, a sequence of $u_{m}$
with discrete measures $\nu_{m}$ weakly converging to a general compactly
supported finite Borel measure $\nu$ has a subsequence converging
to a solution $u$ of the Dirichlet problem.
\begin{lem}
\label{lem:k-convex=000020extension}Let $\Omega$ be $k$-convex.
If and only if it is strictly $k$-convex, any $u_{\partial}\in C(\partial\Omega)$
can be extended to a $k$-convex function $u_{\partial}\in\Phi^{k}(\Omega)\cap C(\bar{\Omega})$
with $\mu_{k}[u_{\partial}]=0$.
\end{lem}

The core of its proof is the following Lemma, which resembles the
theorems of Stone--Weierstraß and Dini.
\begin{lem}
\label{lem:Stone=000020Dini}Let $K$ be a compact topological space
and $M$ be a family of continuous functions on $K$ with the following
properties:
\begin{enumerate}
\item For any $x\in K$, there exists $u\in M$ with a strict maximum in
$x$.
\item For all $u\in M$, $a\in\mathbb{R}$, and $b>0$, also $a+bu\in M$.
\item For all $u,w\in M$, also $\max(u,w)\in M$.
\end{enumerate}
Then $M$ is dense in $C(K)$.

\end{lem}

\begin{proof}
Let $w\in C(K)$ and $\varepsilon>0$ be given. For $x\in K$, let
$u_{x}\in M$ be a function satisfying $u_{x}(x)=0>u(K\setminus\{x\})$
by (1) and (2). Let $U(x)$ be a neighborhood of $x$ where $\sup_{U}(w)-\inf_{U}(w)\leq\varepsilon$.
Choose $a_{x}>0$ sufficiently large such that 
\[
w_{x}:=a_{x}u_{x}+w(x)\leq w\quad\text{outside }U(x).
\]
Inside $U(x)$, this function is bounded from above by $w(x)\leq w+\varepsilon$,
altogether $w_{x}\leq w+\varepsilon$ in $K$. In $x$, we have the
identity $w_{x}(x)=w(x)$ and $W_{x}:=\{|w-w_{x}|<\varepsilon\}$
is a neighborhood of $x$. Due to compactness, there exist finitely
many points $x_{1},\dots,x_{m}$ such that $W_{x_{1}},\dots,W_{x_{m}}$
cover $K$. Then $w_{\varepsilon}:=\max\{w_{x_{1}},\dots,w_{x_{m}}\}\in M$
by (3) and $w-\varepsilon\leq w_{\varepsilon}\leq w+\varepsilon$.
\end{proof}
\begin{proof}[Proof of Lemma \ref{lem:k-convex=000020extension}]
\emph{If $\Omega$ is strictly $k$-convex}, let us show that $K:=\partial\Omega$
and $M$ being the restrictions of $\Phi^{k}(\Omega)\cap C(\bar{\Omega})$
to $\partial\Omega$ satisfy the assumptions of Lemma \ref{lem:Stone=000020Dini}:
(1) by definition of a strict $k$-convex domain, (2) is an obvious
property of $k$-convex functions and (3) by Lemma \ref{lem:max}.
Therefore, Lemma \ref{lem:Stone=000020Dini} implies that there exists
a sequence $u_{N}\in\Phi^{k}(\Omega)\cap C(\bar{\Omega})$ which converges
on $\partial\Omega$ to $u_{\partial}$ uniformly. The comparison
principle implies that 
\[
\left\Vert \left(u_{N}|_{\partial\Omega}\right)_{*}-\left(u_{M}|_{\partial\Omega}\right)_{*}\right\Vert _{\infty}=\left\Vert u_{N}-u_{M}\right\Vert _{C(\partial\Omega)},
\]
i.e.\ that the sequence of these $k$-convex envelopes converges
uniformly, too, and its limit is the wanted extension of $u_{\partial}$.
Corollary \ref{cor:Grenzwerts=0000E4tze} confirms that this limit
is indeed $k$-convex and (\ref{eq:weak*=000020conti}) that $\mu_{k}[u_{\partial}]=0$.

\emph{The necessity of the strict $k$-convexity} for the extendability
can be easily seen as follows: If a function on the boundary with
a strict maximum in $y\in\partial\Omega$ is extended to $\bar{\Omega}$,
the strong maximum principle \cite[Thm.~2.2, p.~15]{GilbargTrudinger2001}
implies that the maximum is also a strict maximum of the extension,
which is, thus, a strong barrier in $y$.
\end{proof}

\section{\protect\label{sec:kreisbeispiel}The Dirichlet problem with discontinuous
boundary values and infinite right-hand sides with sub- and supersolutions}

In this section, we prove Theorem \ref{thm:gedaechtniskirche} and
give examples of convex functions with discontinuous boundary values,
solvable, and an unsolvable Dirichlet problem. 

The following lemma is the key to apply the Perron method to the Hessian
measure.
\begin{lem}
\label{lem:AK=000020Massabsch=0000E4tzung}Let $1\le k\le n$ and
$-\infty<u\le w$ be $k$-convex functions on an open set $\Omega$
with $u\in C(\Omega)$. Then $\mu_{k}[u]\le\mu_{k}[w]$ on $\{u=w\}$.
\end{lem}

\begin{proof}
Since $w-u$ is upper semi-continuous (by definition of a $k$-convex
function and because $u\in C(\Omega)$), $\{u>w{-}\varepsilon\}$
is a neighborhood of $\{u=w\}$ for $\varepsilon>0$. Therefore, $w_{\varepsilon}:=\max(u,w{-}\varepsilon)$
satisfies $w_{\varepsilon}=u$ in a neighborhood of $\{u=w\}$ and
(since $\mu_{k}$ is local) $\mu_{k}[w_{\varepsilon}]=\mu_{k}[u]$
in $\{u=w\}$. According to the weak{*} continuity of $\mu_{k}$ (\ref{eq:weak*=000020conti})
and Lemma \ref{lem:weak*=000020convergence}, a compact set $K\subset\{u=w\}$
satisfies
\[
\mu_{k}[w](K)=\left(\lim_{\varepsilon\to0}\mu_{k}[w_{\varepsilon}]\right)(K)\ge\limsup_{\varepsilon\to0}\left(\mu_{k}[w_{\varepsilon}](K)\right)=\mu_{k}[u](K).
\]
Since $\mu_{k}$ is inner regular (i.e., $\mu_{k}[u](A)=\sup\left\{ \mu_{k}[u](K)~\middle|~K\subset\subset A\right\} $
for all Borel sets $A$), this implies $\mu_{k}[w]\ge\mu_{k}[u]$
on $\{u=w\}$.
\end{proof}
\begin{cor}
\label{cor:max4Perron}For an open set $\Omega$, $1\le k\le n$ and
$u,w\in\Phi^{k}(\Omega)\cap C(\Omega)$, it holds that
\[
\mu_{k}[\max(u,w)]\ge\min\left(\mu_{k}[u],\mu_{k}[w]\right).
\]
\end{cor}

\begin{proof}
Since $\{u<w\}$ is an open set and $\mu_{k}$ is local, $\mu_{k}[\max(u,w)]=\mu_{k}[w]$
there. On $\{u\ge w\}$, the statement holds by Lemma \ref{lem:AK=000020Massabsch=0000E4tzung}.
\end{proof}
This Corollary is the sinews of the Perron method to prove Theorem
\ref{thm:gedaechtniskirche}(\ref{enu:Zwischenwertsatz}).
\begin{proof}[Proof of Theorem \ref{thm:gedaechtniskirche}]
\emph{The uniqueness (\ref{enu:uniqueness}) follows from the comparison
principle (\ref{eq:comparison=000020principle}) and the statement
(\ref{enu:Klassiker}) is just a reformulation of Theorem \ref{thm:Bern}.}

\emph{Proof of (\ref{enu:Zwischenwertsatz}).} Let $u_{1},u_{3}\in\Phi^{k}(\Omega)$
be functions with $u_{i}-u_{\partial}\in C_{0}(\Omega)$ and $\mu_{k}[u_{i}]=\nu_{i}$
for $i=1,3$. We will show that the Perron solution
\[
u_{2}:=\sup W\quad\text{for }W:=\left\{ w\in\Phi^{k}(\Omega)~\middle|~\mu_{k}[w]\ge\nu_{2},u_{3}\le w\le u_{1}\right\} 
\]
solves $\mu_{k}[u_{2}]=\nu_{2}.$ Since $u_{3}\in W$, $W$ is not
empty. Since $W$ is locally uniformly bounded, $u_{2}\in\Phi^{k}(\Omega)$
by Corollary \ref{cor:Grenzwerts=0000E4tze}. Let $\{\tilde{w}_{1},\dots\}$
be a countable subset of $W$ which is locally dense with respect
to $\left\Vert \cdot\right\Vert _{\infty}$ and $w_{N}:=\max(\tilde{w}_{1},\dots,\tilde{w}_{N})$,
which belongs to $W$ by Lemma \ref{lem:max} and Corollary \ref{cor:max4Perron}.
The sequence $w_{N}$ converges to $u_{2}$ pointwise. The weak{*}
continuity (\ref{eq:weak*=000020conti}) shows that the limit $u_{2}=\lim_{N\to\infty}w_{N}$
satisfies $\mu_{k}[u_{2}]\ge\nu_{2}$, too. Now replace $u_{2}$ on
a ball $B\subset\subset\Omega$ by the solution of 
\[
\begin{cases}
\mu_{k}[\tilde{u}_{2}]=\nu_{2} & \text{in }B\\
\tilde{u}_{2}=u_{2} & \text{outside }B
\end{cases}
\]
yielded by Theorem \ref{thm:Bern}. While on the one hand, the comparison
principle shows that $\tilde{u}_{2}\ge u_{2}$, on the other hand,
the definition of $k$-convex functions confirms $\tilde{u}_{2}\in\Phi^{k}(\Omega)$,
Lemma \ref{lem:AK=000020Massabsch=0000E4tzung} shows that $\mu_{k}[\tilde{u}_{2}]\ge\mu_{k}[u_{2}]\ge\nu_{2}$
outside $B$, and (again using the comparison principle to show $\tilde{u}_{2}\le u_{1}$)
that $\tilde{u}_{2}\in W$. But $\tilde{u}_{2}\in W$ implies that
$\tilde{u}_{2}\le u_{2}$ and $\tilde{u}_{2}$ already equals $u_{2}$,
i.e.\ $\mu_{k}[u_{2}]=\nu_{2}$, because $B$ was chosen arbitrarily.
The bounds $u_{3}\le u_{2}\le u_{1}$ guarantee that $u_{2}-u_{\partial}\in C_{0}(\Omega)$. 

\emph{Proof of (\ref{enu:Addition}).} Let $u_{3}:=u_{1}+u_{2}$,
and $w_{i}\in\Phi^{k}(\Omega)$ with $\mu[w_{i}]=\nu_{i}$, $w_{i}-u_{i}\in C_{0}(\Omega)$
for $i=1,2,3$, such that $\nu_{3}\le\nu_{1}+\nu_{2}\le\mu_{k}[w_{1}+w_{2}]$
by assumption and the superadditivity (\ref{eq:superadditivity})
and $w_{3}\ge w_{1}+w_{2}$ by the comparison principle (\ref{eq:comparison=000020principle}).
Then $(\nu_{1}+\nu_{2},u_{1}+u_{2})\in\RHS$ follows from (\ref{enu:Zwischenwertsatz}).
\end{proof}
\begin{example}
Theorem \ref{thm:gedaechtniskirche}(\ref{enu:Addition}) implies
by dint of Lemma \ref{lem:k-convex=000020envelope} that the Dirichlet
problem with continuous nonvanishing boundary values $u_{\partial}\in C(\bar{\Omega})\cap\Phi^{k}(\Omega)$
is solvable if the Dirichlet problem with vanishing boundary values
is and even more: \emph{Let $k>n/2$, $\Omega\subset\mathbb{R}^{n}$
be open, bounded and $k$-convex. Let $u_{\nu},u_{\partial}\in\Phi^{k}(\Omega)\cap C(\bar{\Omega})$
and $\nu\le\mu_{k}[u_{\nu}]$. Then the Dirichlet problem
\begin{equation}
\begin{cases}
\mu_{k}[u]=\nu & \text{in }\Omega,\\
u=u_{\nu}+u_{\partial} & \text{on }\partial\Omega
\end{cases}\label{eq:Randinhomogenisierung}
\end{equation}
has a unique solution.}
\end{example}

On strictly convex domains, the combination of Theorem \ref{thm:gedaechtniskirche}(\ref{enu:Addition})
and Lemma \ref{lem:k-convex=000020extension} imply that the solvability
of a Dirichlet problem is independent of the boundary values:
\begin{cor}
\label{cor:Independence=000020of=000020solutions}Let $k>n/2$, $\Omega\subset\mathbb{R}^{n}$
be open, bounded and strictly $k$-convex. Then $\RHS$ (see Theorem
\ref{thm:gedaechtniskirche}) has the property
\[
(\nu,g_{1})\in\RHS\quad\text{if and only if}\quad(\nu,g_{2})\in\RHS
\]
for any $g_{1},g_{2}\in C(\bar{\Omega})$ and locally finite Borel
measure $\nu$. 
\end{cor}

Finally, let us add some examples of functions with discontinuous
boundary data in order to show that the statements about discontinuous
boundary data in Theorem \ref{thm:gedaechtniskirche} are not redundantly
general.
\begin{example}
\label{exa:kreisbeispiel}Let us consider the circles $\Omega_{r}:=B_{r}((0,r))$
with radius $r$ touching the $y$-axis from the right in 0 and as
a first example the convex envelope $u_{1}$ of $u_{1}(0)=0,u_{1}|_{\partial\Omega_{r}\setminus\{0\}}=r$
for arbitrary $r>0$. We extend the function to
\begin{align*}
u_{1} & :=\left\{ \begin{array}{ll}
\frac{1}{2}\left(\frac{x_{2}^{2}}{x_{1}}+x_{1}\right) & \text{for }x\in\left]0,\infty\right[\times\mathbb{R}\\
0 & \text{for }x=0
\end{array}\right\} \in\Phi^{n}(\left]0,\infty\right[\times\mathbb{R})\setminus C(\partial\Omega_{r}).
\end{align*}
Restricted to $\bar{\Omega}_{r}$, $u_{1}$ is the convex hull of
the boundary data (note that it is affine on rays starting from 0),
thus its Monge--Ampère measure vanishes. Theorem (\ref{thm:gedaechtniskirche})
implies that the Dirichlet problem 
\[
\begin{cases}
\mu_{n}[u]=\nu\\
u-u_{1}\in C_{0}(\Omega)
\end{cases}
\]
is uniquely solvable for any finite Borel measure $\nu$.

Secondly, let
\[
u_{2}:=\begin{cases}
1 & \text{in }\bar{\Omega}_{1}\\
u_{1} & \text{in }\Omega_{2}\setminus\bar{\Omega}_{1}.
\end{cases}
\]
We notice that the convex function $u_{2}$ coincides with $u_{1}/2+1$
on the boundary of $\Omega_{2}$, but $u_{1}/2+1-u_{2}\notin C_{0}(\Omega)$,
because on the dotted circle $\gamma:=\partial\Omega_{1}\setminus\{0\}$,
$u_{2}=1\neq3/2=u_{1}/2+1$ and $\partial\Omega_{1}$ touches $\partial\Omega_{2}$.
Additionally, we want to show that there does not exist a convex function
$u$ with the following properties:
\begin{equation}
\begin{cases}
\mu_{2}[u]\ge\mu_{2}[u_{2}] & \text{in }\Omega_{2},\\
u(x)\to0 & \text{for }x\to x_{0}\in\partial\Omega_{2}\setminus\{0\},\\
\liminf_{}u(x)>-1/2 & \text{for }\gamma\ni x\to0,\\
u\in C^{2}(\bar{\Omega}_{1}\setminus\{0\})\cap C^{2}(\Omega_{2}\setminus\Omega_{1}).
\end{cases}\label{eq:unsolvable_lang}
\end{equation}
In particular, the Dirichlet problem 
\begin{equation}
\{\mu_{2}[u]=\mu_{2}[u_{2}],u\in C_{0}(\Omega)\}\label{eq:unsolvable}
\end{equation}
has no convex solution $u\in C^{2}(\bar{\Omega}_{1}\setminus\{0\})\cap C^{2}(\Omega_{2}\setminus\Omega_{1})$,
which is remarkable compared to Corollary \ref{cor:Independence=000020of=000020solutions}.
We think that the last condition of (\ref{eq:unsolvable_lang}) follows
from the smoothness of the other conditions, but this shall not occupy
us here. To prove the nonexistence, assume that $u$ would be a convex
function satisfying (\ref{eq:unsolvable_lang}). On $\gamma$, we
decompose the inner and the outer gradient of $u$ by 
\begin{align*}
\nabla_{i}u(x) & :=\nabla(u|_{\bar{\Omega}_{1}\setminus\{0\}})(x)=\alpha(x)\tau(x)+\beta_{i}(x)\nu(x)\\
\nabla_{o}u(x) & :=\nabla(u|_{\Omega_{2}\setminus\Omega_{1}})(x)=\alpha(x)\tau(x)+\beta_{o}(x)\nu(x),
\end{align*}
where $\tau$ is the unit tangent vector to $\partial\Omega_{1}$,
$\nu$ its outer normal and $\alpha=\partial_{\tau}u$ and $\beta_{i},\beta_{o}$
are real-valued functions. Then on $\gamma$, the Monge--Ampère measure
of $u$ is given by the volume of the union of the line segments $[\beta_{o}(x),\beta_{i}(x)]\nu(x)+\alpha(x)\tau(x)$,
i.e.
\[
\mu_{2}[u](A)=\int_{A}(\beta_{o}-\beta_{i})\partial_{\tau}\alpha+\frac{1}{2}\left(\beta_{o}^{2}-\beta_{i}^{2}\right)ds
\]
for an arc $A\subset\gamma$. This yields
\[
\mu_{2}[u_{2}]=\frac{1}{2}\beta_{o}^{2}ds=\frac{1}{2}\left|\nabla u_{1}\right|^{2}ds=\frac{ds}{2x_{1}^{2}}\quad\text{on }\partial\Omega_{1}
\]
(and elsewhere, $\mu_{2}[u_{2}]$ vanishes). Now we parametrize $\gamma$
by arc length, starting at $0$, and abbreviate $u(s):=u(1{-}\cos s,\allowbreak\sin s)$
and $\alpha(s)$ and $\beta(s)$ likewise. It holds that $\alpha(s)=u'(s)$
and $\partial_{\tau}\alpha(s)=u''(s)$. This yields the ordinary differential
inequality
\begin{equation}
\frac{1}{2x_{1}^{2}}=\frac{d\mu_{2}[u_{2}]}{ds}\le\frac{d\mu_{2}[u]}{ds}=(\beta_{o}-\beta_{i})u''+\frac{1}{2}\left(\beta_{o}^{2}-\beta_{i}^{2}\right)\le(\beta_{o}-\beta_{i})u''+\frac{1}{2}\beta_{o}^{2}.\label{eq:example=000020ode}
\end{equation}
Since $u$ is convex, $u(x)\le0$, $\beta_{o}\ge\beta_{i}$ and $\nabla_{o}u(x)$
and $\nabla_{i}u(x)$ have to satisfy $u\ge u(x)+\nabla_{*}u(x)$
for $*=i,o$. Applied to the points $y_{*}=x+\lambda_{*}\nu(x)$ on
$\partial\Omega_{2}$ with $\lambda_{i}\leq0\le\lambda_{o}$, this
means that $\lambda_{*}\beta_{*}(x)\le-u(x)$. It is easy to show
that the distances $\lambda_{o}=|y_{o}-x|\in x_{1}/2+o(x_{1})$ and
$\lambda_{i}=-\left|y_{i}-x\right|\in-O(1)$ for $\partial\Omega_{1}\ni x\to0$
and therefore $\beta_{o}\le-2u/x_{1}(1+o(1))$, $\beta_{i}\gtrsim u$
near $0$ (where $f\gtrsim g$ means that there exists a constant
$c>0$ such that $f\ge c$ for all $s$). Plugged into (\ref{eq:example=000020ode}),
this yields
\begin{align*}
u''(s) & \ge\frac{1}{2}\frac{x_{1}^{-2}-\beta_{o}^{2}}{\beta_{o}-\beta_{i}}\gtrsim\frac{1-4u^{2}}{-x_{1}u}.
\end{align*}
On the dotted circle $\gamma$ it holds that $x_{1}(s)\le s^{2}$
near $0$. If now $u(s)\ge c$ for a constant $c>-1/2$ and all sufficiently
small $s$, then 
\[
u''(s)\gtrsim1/x_{1}(s)\ge s^{-2}
\]
which has no bounded solution.
\end{example}

\section{\protect\label{sec:Open-questions}Open questions}

Our main results are open for $k\leq n/2$.

Is Lemma \ref{lem:k-convex=000020envelope} extendable to lower-semicontinuous
discontinuous boundary values? Let us elaborate the question: For
simplicity, let $\Omega$ be a bounded convex domain and $u_{\partial}\in\Phi^{k}(\Omega)$
be bounded but discontinuous on $\bar{\Omega}$. Let us denote the
lower semi-continuous envelope of $u$ by $\underline{u}$. Does the
Perron solution
\[
u:=\sup\left\{ w\in\Phi^{k}(\Omega)~\middle|~\underline{w}\le\underline{u}_{\partial}\right\} 
\]
(here lower and not upper semi-continuous envelopes!) solve the associated
Dirichlet problem
\[
\begin{cases}
\mu_{k}[u]=0 & \text{in }\Omega\\
\underline{w}=\underline{u}_{\partial} & \text{on }\partial\Omega?
\end{cases}
\]
This is easily seen for convex functions, where the convex envelope
of the boundary data is well-defined.

Does Corollary \ref{cor:Independence=000020of=000020solutions} (``The
solvability is independent of the boundary values.'') hold also for
nonstrict $k$-convex domains (provided the boundary values stem from
a $k$-convex function) or bounded discontinuous boundary values?
This question seems to be undiscussed even for the Monge--Ampère
equation. In the case of discontinuous boundary values, of course,
the answer depends on the definition, how the boundary conditions
have to be satisfied. Let us discuss Example \ref{exa:kreisbeispiel}
a little more from this point of view. Our boundary condition $u{-}u_{\partial}\in C_{0}(\Omega)$
(\ref{eq:DP=000020discontinuous}) was fitted for the comparison principle,
but (\ref{eq:unsolvable}) shows an \emph{unsolvable} Dirichlet problem
(probably: if the differentiability condition can be dispensed with)
with a measure coming from a \emph{solvable} Dirichlet problem with
a bounded right-hand side. So we may ask:
\begin{problem}
Is there a weaker notion of boundary conditions instead of $u{-}u_{\partial}\in C_{0}(\Omega)$
such that every Dirichlet problem of a wider class than $\RHS$ from
Theorem \ref{thm:gedaechtniskirche} has a unique solution and the
solvability is independent of the boundary values, provided they are
bounded?
\end{problem}

One can prescribe a lower semi-continuous function $g$ on the boundary
and demand $\lim\inf_{\Omega\ni x\to y\in\partial\Omega}u(x)=g(y)$
(in that sense $u_{2}$ coincides on the boundary with $u_{1}/2+1$,
so this enlarges the class of solvable problems probably), but (\ref{eq:unsolvable})
stays unsolvable in that sense. Another notion of boundary condition
could demand even less, e.g.\ that $u(x){-}u_{\partial}(x)\to0$
for any $x\to x_{0}$ only for \emph{almost every }$x_{0}\in\partial\Omega$.
Assume that the Dirichlet problem $\{\mu_{2}[u]=\mu_{2}[u_{2}],u=0\text{ on }\partial\Omega_{2}\}$
would have a solution in such a weak sense. Would not $0$ and $u_{1}$
both be weak solutions to the Dirichlet problem $\{\mu_{2}[w]=0,w=0\text{ on }\partial\Omega_{2}\}$
then, violating the uniqueness of solutions?

\appendix

\section{\protect\label{sec:Theorems-from-others}Cited results}
\begin{defn}[{Principal curvature, from \cite[Sec.~14.6]{GilbargTrudinger2001}}]
\label{def:principals}%
Let $\Omega$ be a $C^{2}$ domain. For $y\in\partial\Omega$, let
$\nu(y)$ and $T(y)$ denote the \emph{inner} normal and the tangent
hyperplane to $\partial\Omega$ at $y$. By a rotation of coordinates,
we can assume that the $x_{n}$ coordinate axis lies in the direction
of $\nu(y)$. In some neighborhood $\mathcal{N}$ of $y_{}$, $\partial\Omega$
is given by $x_{n}=\phi(x')$ where $x'=(x_{1},\dots,x_{n-1})$, $\phi\in C^{2}(T(y)\cap\mathcal{N})$
and $D\phi(y_{})=0$. The eigenvalues $\kappa_{1},\dots,\kappa_{n-1}$
of the Hessian $D^{2}(\phi(y'_{}))$ are called the \emph{principal
curvatures }of $\partial\Omega$ at $y$.
\end{defn}

\begin{lem}[{Definition of weak{*} convergence of measures, cf. \cite[Thm.~1.40, Sec.~1.9, p.~65]{Evans_measure}}]
\label{lem:weak*=000020convergence}The following statements are
equivalent definitions of weak{*}-convergence of locally finite Borel
measures $\mu_{j}$ on $\Omega$:
\begin{enumerate}
\item $\int_{\Omega}\phi\,\d\mu_{j}\to\int_{\Omega}\phi\,\d\mu$ for any
compactly bounded $\phi\in C_{c}(\Omega)$.
\item $\limsup_{j\to\infty}\mu_{j}(K)\le\mu(K)$ for any $K\subset\Omega$
and $\liminf_{j\to\infty}\mu_{j}(U)\ge\mu(U)$ for any $U\subset\Omega$.
\end{enumerate}
\end{lem}

Caffarelli, Nirenberg, and Spruck proved Lemma \ref{thm:k-1=000020konvex}
and its converse with the additional attributes ``smooth'' and ``uniform'',
i.e.:
\begin{lem}[{\cite[Thm.~3, see also Prop.~1.3, Rem.~1.1, p.~270]{CaffarelliNirenbergSpruck3}}]
\label{rem:uniform=000020domain}A \emph{smooth} open and bounded
domain $\Omega$ has a\emph{ uniform }$\Gamma_{k-1}$-boundary if
and only if there exists a \emph{uniformly} $k$-convex function $u\in C^{\infty}(\bar{\Omega})$
with $\Omega=\{u<0\}$.
\end{lem}

In \cite[§8]{WangHessian}, an equivalent definition of $k$-convex
functions (he calls them $k$-admissible) is used. To make it transparent
that we can use the results of \cite[§8]{WangHessian}, we state
the equivalence of the two definitions by the following Lemma. The
dual cones are given by
\[
\Gamma_{k}^{*}:=\left\{ \lambda\in\mathbb{R}^{n}~\middle|~\langle\lambda,\nu\rangle>0\text{ for all }\nu\in\Gamma_{k}\right\} .
\]

\begin{lem}[{equals \cite[Lem.~2.2]{TrudingerWangHM2}}]
A distribution $T$ on $\Omega$ is equivalent to a $k$-convex function
in $\Omega$ if and only if
\[
T(a^{ij}D_{ij}\phi)\geq0
\]
for all $\phi\geq0$, $\phi\in C_{c}^{\infty}(\Omega)$ and for all
constant symmetric matrices $A=a^{ij}$ with eigenvalues $\lambda(A)\in\Gamma_{k}^{*}$.
\end{lem}

\begin{lem}[{\cite[Lem.~8.1]{WangHessian}}]
\label{lem:Grenzwert=000020k-convex}Let $u_{j}$ be a sequence of
$k$-convex functions which converges to an upper semi-continuous
function $u$ almost everywhere and $\{u=-\infty\}$ has measure zero.
Then $u$ is $k$-convex and $u_{j}$ converges to $u$ pointwise.
\end{lem}

\begin{lem}[{Concatenation, equals \cite[Lem.~2.5]{TrudingerWangHM2}}]
\label{lem:max}Let $u_{1},\dots,u_{m}\in\Phi^{k}(\Omega)$ and $f$
be a convex, nondecreasing function in $\mathbb{R}^{m}$. Then the
composite function $w=f(u_{1},\dots,u_{m})\in\Phi^{k}(\Omega)$ also.
\end{lem}

In particular, the maximum of two $k$-convex functions is $k$-convex.
\begin{lem}[{Monotonicity, equals \cite[ Cor.~2.4]{TrudingerWangHM1}}]
\label{lem:monotonicity}If $k>n/2$, $\Omega$ open and bounded,
and $u,w\in\Phi^{k}(\Omega)\cap C(\bar{\Omega})$ satisfy $u=w$ on
$\partial\Omega$ and $u\le w$ in $\Omega$, then
\[
\mu_{k}[u](\Omega)\geq\mu_{k}[w](\Omega).
\]
\end{lem}

\begin{lem}[{Local Hölder continuity, cf. \cite[(3.6)]{TrudingerWangHM1}}]
\label{lem:local=000020H=0000F6lder=000020continuity}Let $k>n/2$
and $\Omega$ open. A function $u\in\Phi^{k}(\Omega)$ and $x,y\in\Omega$
satisfy the local Hölder continuity
\[
\frac{|u(x)-u(y)|}{\left|x{-}y\right|^{\alpha}}\le\frac{\osc u}{\min(\dist(x,\partial\Omega),\dist(y,\partial\Omega))^{\alpha}}
\]
with $\alpha=2-n/k$.
\end{lem}

Since a locally uniformly Hölder continuous set of functions is also
locally equicontinuous, the Arzelà--Ascoli theorem implies the following
well-known consequences:
\begin{cor}
\label{cor:Grenzwerts=0000E4tze}Let $k>n/2$ and $\Omega$ be open.
\begin{enumerate}
\item A locally bounded set of $k$-convex functions on $\Omega$ is relatively
compact. 
\item If a locally uniformly bounded sequence of $k$-convex functions on
$\Omega$ converges pointwise, it converges locally uniformly to a
$k$-convex function on $\Omega$.
\item The supremum of a nonempty locally uniformly bounded set of $k$-convex
functions on $\Omega$ is $k$-convex.
\end{enumerate}
\end{cor}

\begin{proof}
(1) is already proved. (2) follows from (1) and Lemma \ref{lem:Grenzwert=000020k-convex}.
To prove (3), note that the supremum of a locally equicontinuous set
of functions satisfies the same equicontinuity. Therefore, it is continuous
and its $k$-convexity is the well-known property that the upper semi-continuous
envelope of the supremum of a set of viscosity subsolutions, if finite,
is a viscosity subsolution as well \cite[Lem.~4.2, p.~23]{CrandallIshiiHitoshiUsersGuide}.
\end{proof}

\section*{Acknowledgements}

The author wishes to thank his professor Dietmar Gallistl for a profound
restructuring of the article, besides many corrections and his cordial
encouragement.

\section*{Disclosure of interest}

The author reports there are no competing interests to declare.

\bibliographystyle{alpha}
\bibliography{literatur}

\end{document}